\documentclass[11pt]{amsart}

\usepackage{amssymb,amsmath,amsthm,newlfont}
\usepackage[shortlabels]{enumitem}

\theoremstyle{plain}
\newtheorem{theorem}{Theorem}[section]
\newtheorem{proposition}[theorem]{Proposition}
\newtheorem{corollary}[theorem]{Corollary}
\newtheorem{lemma}[theorem]{Lemma}
\newtheorem{conjecture}[theorem]{Conjecture}

\theoremstyle{definition}

\theoremstyle{remark}

\newcommand{\CC}{\mathbb{C}}
\newcommand{\DD}{\mathbb{D}}

\newcommand{\TT}{\mathbb{T}}
\newcommand{\cD}{\mathcal{D}}
\newcommand{\cF}{\mathcal{F}}

\newcommand{\cI}{\mathcal{I}}
\newcommand{\cK}{\mathcal{K}}
\newcommand{\cL}{\mathcal{L}}
\DeclareMathOperator{\rank}{rank}
\DeclareMathOperator{\dist}{dist}
\DeclareMathOperator{\supp}{supp}

\begin{document}


\title[Negative powers of contractions]{Negative powers of Hilbert-space contractions}

\author[T. Ransford]{Thomas Ransford}
\address{D\'epartement de math\'ematiques et de statistique, Universit\'e Laval,
Qu\'ebec City (Qu\'ebec),  Canada G1V 0A6.}
\email{ransford@mat.ulaval.ca}

\dedicatory{To the memory of Jean Esterle}

\thanks{Research supported by grants from NSERC and the Canada Research Chairs program.}

\begin{abstract}
We show that, given a closed subset $E$ of the unit circle of Lebesgue measure zero,
there exists a positive sequence $u_n\to\infty$ with the following property:
if $T$ is a Hilbert-space contraction such that $\sigma(T)\subset E$ and $\|T^{-n}\|=O(u_n)$
and $\rank(I-T^*T)<\infty$, then $T$ is a unitary operator. We further show that
the condition of measure zero is sharp.
\end{abstract}

\keywords{Contraction, Hilbert space, inner function, compressed shift, functional model}

\makeatletter
\@namedef{subjclassname@2020}{\textup{2020} Mathematics Subject Classification}
\makeatother

\subjclass[2020]{Primary 30J15, 47A20; Secondary 47A30,  47A45, 47A56}

\maketitle

\section{Introduction}\label{S:intro}

Throughout this article, by \emph{operator} 
we mean a bounded linear operator on a complex Banach space $X$.
The Banach algebra of all bounded operators on $X$ is denoted by $\cL(X)$.
An operator $T$ is a \emph{contraction} if $\|T\|\le1$. 
We denote the spectrum of $T$ by $\sigma(T)$.

Our story begins with the following result of Zarrabi \cite{Za93},
which was based on earlier work of Atzmon \cite{At80}.

\begin{theorem}[\protect{\cite[Th\'eor\`eme~3.1]{Za93}}]\label{T:Zarrabi}
Let $T$ be a Banach-space contraction such that 
$\sigma(T)$ is a countable subset of the unit circle $\TT$
and 
\begin{equation}\label{E:Zarrabi}
\lim_{n\to\infty}\frac{\log\|T^{-n}\|}{\sqrt{n}}=0.
\end{equation}
Then $T$ is an isometry, i.e., $\|T\|=\|T^{-1}\|=1$.
\end{theorem}

Zarrabi further showed that the condition that $\sigma(T)$ be countable 
is sharp in the following sense:
given an uncountable closed subset $E$ of $\TT$, 
there exists a contraction $T$ on a suitably chosen Banach space
such that $\sigma(T)\subset E$ 
and such that \eqref{E:Zarrabi} holds, but $\sup_n\|T^{-n}\|=\infty$.

This suggests that, to treat the case where $\sigma(T)$ is uncountable, 
we need to adapt the condition \eqref{E:Zarrabi}.
The next result, due to Esterle \cite{Es23} does just that.

Let $A(\TT)$ and $A^+(\TT)$ denote the algebras 
of absolutely convergent Fourier and  Taylor series respectively.
Given a closed subset $E$ of $\TT$, 
we denote by $J(E)$ the closure of the set of functions in $A(\TT)$ vanishing
on a neighbourhood of $E$. 
We say that $E$ is a \emph{strong $AA^+$-set} if, 
for every $f\in A(\TT)$, there exists $g\in A^+(\TT)$ such that $f-g\in J(E)$. 
Examples of strong $AA^+$-sets include closed countable subsets of $\TT$ 
and the Cantor middle-thirds set in $\TT$.

\begin{theorem}[\protect{\cite[Theorem~3.3]{Es23}}]\label{T:EsterleBanach}
Let $E$ be a strong $AA^+$-subset of $\TT$. 
Then there exists a positive sequence $u_n\to\infty$ and a constant $K$
with the following property:
if $T$ is a Banach-space contraction  such that $\sigma(T)\subset E$ and 
$\|T^{-n}\|=O(u_n)$ as $n\to\infty$, then $\sup_n\|T^{-n}\|\le K$.
\end{theorem}

The condition that $E$ be a strong $AA^+$-set may seem restrictive, 
but in fact it is optimal.
Esterle had previously shown in \cite{Es94b} that, 
given a closed subset $E$ of $\TT$ that is not a strong $AA^+$-set,
and given any positive sequence $u_n\to\infty$,
there exists a contraction $T$ on a suitably chosen Banach space such that 
$\sigma(T)\subset E$ and
$\|T^{-n}\|=O(u_n)$, but $\sup_n\|T^{-n}\|=\infty$.

The situation is quite different for contractions on \emph{Hilbert spaces}.
In a private communication,
Esterle expressed the belief that the following result is true. 
We formulate it as a conjecture.

\begin{conjecture}\label{Cj:Esterle}
Let $E$ be a closed subset of $\TT$ of Lebesgue measure zero.
Then there exists a positive sequence $u_n\to\infty$ with the following property:
if $T$ is a Hilbert-space contraction such that $\sigma(T)\subset E$ and $\|T^{-n}\|=O(u_n)$
as $n\to\infty$, then $T$ is a unitary operator.
\end{conjecture}

The conjecture is known to be true if $E$ is a so-called \emph{Beurling--Carleson set}, 
namely a closed subset of $\TT$ such that
\[
\int_\TT \log\dist(e^{i\theta},E)\,d\theta>-\infty.
\]
In fact Esterle established the following result.

\begin{theorem}[\protect{\cite[Theorem~6.4]{Es94a}}]\label{T:Esterle}
Let $E$ be a Beurling--Carleson set.
If $T$ is a Hilbert-space contraction such that $\sigma(T)\subset E$ and 
$\|T^{-n}\|=O(n^k)$ for some $k\ge0$,
then $T$ is a unitary operator.
\end{theorem}

Esterle also showed that Theorem~\ref{T:Esterle} does not extend to
general closed sets $E$ of measure zero. In fact, by a result of Kellay \cite[Th\'eor\`eme 3.4]{Ke98},
given any positive sequence $u_n\to\infty$, there exists a Hilbert-space contraction $T$
such that $\sigma(T)$ is a subset of $\TT$ of measure zero and $\|T^{-n}\|=O(u_n)$,
but $\sup_n\|T^{-n}\|=\infty$.
Thus, in Conjecture~\ref{Cj:Esterle},
no single sequence $(u_n)$ can work simultaneously for all sets $E$ of measure zero.

Esterle wrote in \cite{Es94a} `The subject clearly deserves further investigation'.
Our main goal in this paper is to establish the following result.

\begin{theorem}\label{T:main}
Let $E$ be a closed subset of $\TT$ of Lebesgue measure zero.
Then there exists a positive sequence $u_n\to\infty$ with the following property:
if $T$ is a Hilbert-space contraction such that $\sigma(T)\subset E$ and $\|T^{-n}\|=O(u_n)$
and  $\rank(I-T^*T)<\infty$, then $T$ is a unitary operator.
\end{theorem}

This result stops short of proving Conjecture~\ref{Cj:Esterle} because of the extra
condition $\rank(I-T^*T)<\infty$. 
However, Theorem~\ref{T:main} does apply to every closed subset $E$ of $\TT$ of measure zero,
the first such result as far as we know.

The condition of measure zero is sharp. This is our second result.

\begin{theorem}\label{T:converse}
Let $E$ be a closed subset of $\TT$ of positive Lebesgue measure.
Then, for any positive sequence $u_n\to\infty$,
there exists a Hilbert-space contraction $T$ such that $\sigma(T)\subset E$ and $\|T^{-n}\|=O(u_n)$
and  $\rank(I-T^*T)=1$, but $\sup_n\|T^{-n}\|=\infty$.
\end{theorem}

The proofs of Theorems~\ref{T:main} and \ref{T:converse} both proceed via a study
of compressed shift operators. 
In \S\ref{S:compressedshifts} we recall the basic properties of these operators,
and derive estimates for the norms of their negative powers. These estimates depend on the
decay properties of singular inner functions, which we study in detail in \S\ref{S:inner}.
These results are then applied in \S\ref{S:return}, both to prove Theorem~\ref{T:converse}
and to obtain Theorem~\ref{T:main} for the case of compressed shifts. 
The general case of Theorem~\ref{T:main} is proved using similar ideas, 
based on the Sz-Nagy--Foias theory of functional models. This theory is presented
in \S\ref{S:functmodels}, 
and then applied in \S\ref{S:mainproof} to complete the proof Theorem~\ref{T:main}.
We conclude in \S\ref{S:conclusion} with a brief discussion of what
would be needed to obtain a complete proof of Conjecture~\ref{Cj:Esterle}.


\section{Compressed shifts}\label{S:compressedshifts}

Let $H^2$ be the classical Hardy space on the unit disk $\DD$, namely the Hilbert space of holomorphic functions 
$f(z)=\sum_{n\ge0}a_nz^n$ such that $\|f\|_{2}^2:=\sum_{n\ge0}|a_n|^2<\infty$. 
We denote by $S:H^2\to H^2$ the shift operator, defined by $Sf(z):=zf(z)$. 

Throughout this section, $\theta$ denotes a (non-constant) inner function on $\DD$, 
namely a bounded holomorphic function whose radial boundary values satisfy  
$|\theta|=1$ a.e.\ on $\TT$.
We write $K_\theta:=H^2\ominus\theta H^2$,
the orthogonal complement of $\theta H^2$ in $H^2$,
and denote by $P_{K_\theta}$  the orthogonal projection of $H^2$ onto $K_\theta$.

The operator $S_\theta:=P_{K_\theta}S|_{K_\theta}$ 
is called the \emph{compressed shift} on $K_\theta$.
In this section, we study some of its  properties. 
To the best of our knowledge, this operator was first
studied by Sarason in \cite{Sa67}.
Our treatment mostly follows that in Nikol'ski\u{\i}'s book \cite{Ni86}.


\subsection{Basic results}

We begin with the following result, 
which is an $H^\infty$-functional calculus for $S_\theta$. 
Here $H^\infty$ denotes the algebra of bounded holomorphic functions on~$\DD$, 
equipped with the norm $\|f\|_{\infty}:=\sup_{z\in\DD}|f(z)|$. 

\begin{theorem}\label{T:functcalc}
For $\phi\in H^\infty$, define the operator $\phi(S_\theta): K_\theta\to K_\theta$ by
\[
\phi(S_\theta)f:=P_{K_\theta}(\phi f) \quad(f\in K_\theta).
\]
Then the map $\phi\mapsto\phi(S_\theta)$ is a contractive homomorphism 
from $H^\infty$ into $\cL(K_\theta)$ such that $z\mapsto S_\theta$. 
Moreover, $\phi(S_\theta)=0$ if and only if $\phi\in\theta H^\infty$.
\end{theorem}

\begin{proof}
See \cite[Chapter~III, \S2, p.64]{Ni86}.
\end{proof}

The following result,
originally due to Sarason \cite{Sa67},
 gives a precise value for the operator norm of $\phi(S_\theta)$.

\begin{theorem}\label{T:norm}
For $\phi\in H^\infty$, we have
\begin{equation}\label{E:norm}
\|\phi(S_\theta)\|=\inf_{h\in H^\infty} \|\phi+\theta h\|_\infty.
\end{equation}
Moreover, the infimum is attained.
\end{theorem}

\begin{proof}
See \cite[Proposition~2.1]{Sa67} or \cite[Lecture~VIII, p.182]{Ni86}.
\end{proof}

The next result is known as the Liv\v sic--Moeller theorem.

\begin{theorem}\label{T:LM}
The spectrum and point spectrum of $S_\theta$ are given respectively by
\[
\sigma(S_\theta)=Y_\theta\cup Z_\theta
\quad\text{and}\quad
\sigma_p(S_\theta)=Z_\theta,
\]
where $Y_\theta$ is the set of singular points of $\theta$ in $\TT$, 
and $Z_\theta$ is the set of zeros of $\theta$ in $\DD$.
\end{theorem}

\begin{proof}
See \cite[Lecture~III, \S1, p.62]{Ni86}.
\end{proof}

We conclude this subsection with two results on the adjoint $S_\theta^*$ of $S_\theta$.

\begin{theorem}\label{T:adjoint}
For all $f\in K_\theta$, we have $\|S_\theta^{*n}f\|_2\to0$ as $n\to\infty$.
\end{theorem}

\begin{proof}
A standard calculation(see e.g.\ \cite[p.62]{Ni86}) shows that $S_\theta^*=S^*|_{K_\theta}$,
where $S^*$ is the $H^2$-adjoint of $S$. Hence, for each $f\in K_\theta$,
we have $\|S_\theta^{*n}f\|_2=\|S^{*n}f\|_2\to0$ as $n\to\infty$.
\end{proof}

\begin{theorem}\label{T:defect}
The operator $I-S_\theta^*S_\theta$ has rank one.
\end{theorem}

\begin{proof}
See e.g.\ \cite[Corollary~1.5]{Be88}.
\end{proof}


\subsection{Norms of negative powers}

If $\theta(0)\ne0$, then, by Theorem~\ref{T:LM}, the operator $S_\theta$ is invertible.
We are interested in computing or estimating the norms of negative powers of $S_\theta$.
We begin with a very simple result.

\begin{theorem}\label{T:negpowers}
We have $\|S_\theta^{-n}\|\to\infty$ as $n\to\infty$.
\end{theorem}

\begin{proof}
By Theorem~\ref{T:adjoint}, $S_\theta^{*n}\to0$ strongly.
It follows that $\|(S_\theta^*)^{-n}\|\to\infty$, and hence that $\|S_\theta^{-n}\|\to\infty$.
\end{proof}

The exact value of $\|S_\theta^{-n}\|$,
and more generally that of $\|\phi(S_\theta)^{-1}\|$ 
for $\phi\in H^\infty$, can be expressed 
in terms of the solution to a corona-type problem.
The connection between the spectral properties of
the compressed shift and the corona theorem is not new:
it appears for example in \cite[Lecture~III]{Ni86}.
We establish the following result.

\begin{theorem}\label{T:corona}
Let $\phi\in H^\infty$, and suppose that $\phi(S_\theta)$ is invertible. Then
\begin{equation}\label{E:corona}
\|\phi(S_\theta)^{-1}\|=\inf\Bigl\{\|g\|_\infty: g,h\in H^\infty, \phi g+\theta h=1\Bigr\}.
\end{equation}
Moreover, the infimum is attained.
\end{theorem}

\begin{proof}
Suppose first that $g,h\in H^\infty$ with $\phi g+\theta h=1$. 
By the functional calculus, Theorem~\ref{T:functcalc}, 
we have $\phi(S_\theta) g(S_\theta)= g(S_\theta)\phi(S_\theta)=I$. Hence
\[
\|\phi(S_\theta)^{-1}\|=\|g(S_\theta)\|\le\|g\|_\infty.
\]

Now, as $\phi(S_\theta)$ is invertible,
it follows by the spectral mapping theorem in \cite[Lecture~III, p.66]{Ni86} that $\inf_{\DD}(|\phi|+|\theta|)>0$.
By the corona theorem, there exist $g_0,h_0\in H^\infty$
such that $\phi g_0+\theta h_0=1$.
By a normal-family argument, there exists $f_0\in H^\infty$
such that $\|g_0\|_{H^\infty/\theta H^\infty}=\|g_0-\theta f_0\|_\infty$.
Set 
\[
g:=g_0-\theta f_0
\quad\text{and}\quad
h:=h_0+\phi f_0.
\]
Then $g,h\in H^\infty$ and
$\phi g+\theta h=1$, so as before $\phi(S_\theta)^{-1}=g(S_\theta)$.
Also, since $\|g\|_{H^\infty/\theta H^\infty}=\|g\|_\infty$, Theorem~\ref{T:norm} gives
\[
\|\phi(S_\theta)^{-1}\|=\|g(S_\theta)\|=\|g\|_{H^\infty/\theta H^\infty}=\|g\|_\infty.
\]
Thus the infimum in \eqref{E:corona} is attained.
\end{proof}

Using this result, we derive some quantitative estimates for $\|\phi(S_\theta)^{-1}\|$.

\begin{theorem}\label{T:estimates}
Let $\phi\in H^\infty$. If $\phi(S_\theta)$ is invertible and $\|\phi\|_\infty\le 1$, then
\[
\frac{1}{2}\Bigl(\frac{1}{\delta}-1\Bigr)\le \|\phi(S_\theta)^{-1}\|\le A\frac{\log(2/\delta)}{\delta^2},
\]
where $\delta:=\inf_{\DD}\max\{|\phi(z)|,|\theta(z)|\}$ and $A$ is an absolute constant.
\end{theorem}

\begin{proof}
We begin with the lower bound. 
By Theorem~\ref{T:corona}, 
there exist $g,h\in H^\infty$ with
$\phi g+\theta h=1$ such that $\|\phi(S_\theta)^{-1}\|=\|g\|_\infty$. 
Now, for all $z\in\DD$, we have
\begin{align*}
1&=|\phi(z) g(z)+\theta(z)h(z)|\\
&\le |\phi(z)|\|g\|_\infty+|\theta(z)|\|h\|_\infty\\
&\le \max\{|\phi(z)|,|\theta(z)|\}(\|g\|_\infty+\|h\|_\infty).
\end{align*}
Note also that $\|h\|_\infty\le\|g\|_\infty+1$. Indeed
\[
\|h\|_\infty=\|\theta h\|_\infty=\|1-\phi g\|_\infty\le\|1\|_\infty+\|\phi g\|_\infty\le 1+\|g\|_\infty,
\]
the last inequality because $\|\phi\|_\infty\le1$.
It follows that 
\[
1\le \max\{|\phi(z)|,|\theta(z)|\}(2\|g\|_\infty+1) \quad(z\in\DD).
\]
Taking the infimum of the right-hand side over all $z\in\DD$, we obtain
\[
1\le \delta(2\|g\|_\infty+1),
\]
where $\delta:=\inf_{\DD}\max\{|\phi(z)|,|\theta(z)|\}$.
This gives the lower bound in the theorem.
 
For the upper bound, we invoke a quantitative form of the corona theorem.
According to that theorem, if $a,b\in H^\infty$ and  $\epsilon\le \max\{|a(z)|,|b(z)|\}\le 1$ for all $z\in\DD$,
then there exist $g,h\in H^\infty$ with $ag+bh=1$ 
and $\|g\|_\infty,\|h\|_\infty\le A\epsilon^{-2}\log(2/\epsilon)$,
where $A$ is an absolute constant (see e.g.\ \cite[Appendix 3, p.288]{Ni86}). 
Applying this result, we deduce
that there exist $g,h\in H^\infty$ with $\phi g +\theta h=1$ 
and $\|g\|_\infty,\|h\|_\infty\le A\delta^{-2}\log(2/\delta)$.
Now apply Theorem~\ref{T:corona}.
\end{proof}

We record the special case of
Theorem~\ref{T:estimates} in which $\phi(z)=z^n$.
Note once again that, by Theorem~\ref{T:LM}, the condition that
$\theta(0)\ne0$ implies that $S_\theta$
is invertible.

\begin{corollary}\label{C:estimates}
If $\theta(0)\ne0$, then, for all $n\ge1$,
\begin{equation}\label{E:estimates}
\frac{1}{2}\Bigl(\frac{1}{\delta_n(\theta)}-1\Bigr)\le \|S_\theta^{-n}\|\le A\frac{\log(2/\delta_n(\theta))}{\delta_n(\theta)^2},
\end{equation}
where 
\[
\delta_n(\theta):=\inf_{z\in\DD}\max\{|z|^n,|\theta(z)|\},
\]
and $A$ is an absolute constant.
\end{corollary}


\section{Decay of singular inner functions}\label{S:inner}

Corollary~\ref{C:estimates} shows that, 
if $\theta$ is an inner function with $\theta(0)\ne0$, then the growth of $\|S_\theta^{-n}\|$ 
is essentially governed by the decay of the sequence
$\delta_n(\theta)$.
In this section we examine the sequence $\delta_n(\theta)$  in detail.

If there exists $w\in\DD\setminus\{0\}$ such that $\theta(w)=0$,
then clearly $\delta_n(\theta)\le |w|^n$ for all $n$. 
Thus, in this case, $\delta_n(\theta)$ decays exponentially fast.

Henceforth we shall be concerned with the more delicate case where $\theta$ has no zeros, 
i.e., it is a singular inner function. 
Thus it has the form
\begin{equation}\label{E:sing}
\theta(z)=\exp\Bigl(-\int_\TT\frac{e^{it}+z}{e^{it}-z}\,d\nu(e^{it})\Bigr)
\quad(z\in\DD),
\end{equation}
where $\nu$ is a finite positive singular measure on $\TT$.
Define 
\[
m_\theta(r):=\inf_{|z|=r}|\theta(z)| \quad(0\le r<1).
\]
Then $m_\theta(r)$ is a strictly decreasing function such that $m_\theta(0)=|\theta(0)|$
and $m_\theta(r)\to0$ as $r\to 1^-$.
Clearly
\begin{equation}\label{E:deltanest}
\delta_n(\theta)\ge r^n \iff m_\theta(r)\ge r^n.
\end{equation}

The following result shows that $\delta_n(\theta)$ can
decay arbitrarily slowly.
We denote by $|\cdot|$ the normalized Lebesgue measure on $\TT$.

\begin{theorem}\label{T:slowdecay}
Let $E$ be a  closed subset of $\TT$ such that $|E|>0$,
and let $(\eta_n)_{n\ge1}$ be a sequence in $(0,1)$ such that $\eta_n\to0$.
Then there exists a finite positive singular measure $\nu$ on $E$ such that,
if $\theta$ is the singular inner function defined by \eqref{E:sing},
then $\delta_n(\theta)\ge\eta_n$ for all $n\ge1$.
\end{theorem}

\begin{proof}
Choose a strictly increasing sequence $(\rho_n)_{n\ge1}\subset(0,1)$
such that $\rho_n>\eta_n^{1/n}$ for each $n$.
By the main result of \cite{Ra21}, there exists a
non-constant singular inner function $\theta$ such that
$m_\theta(\rho_n)\ge\eta_n$ for all $n$.
Hence $m_\theta(\eta_n^{1/n})\ge \eta_n$ for all~$n$.
From \eqref{E:deltanest}, it follows that $\delta_n(\theta)\ge \eta_n$
for all $n$.

Although it is not explicitly mentioned in \cite{Ra21},
it follows easily from the proof therein that $\theta$ may be chosen
so that its measure $\nu$ is supported on a given
closed subset $E$ of $\TT$, provided that $|E|>0$.
\end{proof}

Theorem~\ref{T:slowdecay} is no longer true if $|E|=0$. 
Indeed, in this case,
$\delta_n(\theta)$ must respect a certain rate of decay.
This is made precise by the following result.

\begin{theorem}\label{T:key}
Let $E$ be a closed subset of $\TT$ such that $|E|=0$.
Then there exists a positive sequence $\epsilon_n\to0$
such that, if $\nu$ is any (non-zero) finite positive measure 
on $E$, and $\theta$ is the singular inner function defined by
\eqref{E:sing}, then
\[
\liminf_{n\to\infty}\frac{\delta_n(\theta)}{\epsilon_n}=0.
\]
\end{theorem}

To prove this result, we need to develop some ideas from
the theory of Hausdorff measures. 

A \emph{measure function} is a continuous increasing function 
$h:(0,1]\to(0,\infty)$ such that $h(t)\to0$ as $t\to0^+$.
Given a measure function $h$, a subset $A\subset\TT$ and $\eta>0$,
we define
\[
\Lambda_h^\eta(A):=\inf\Bigl\{\sum_{I\in\cI}h(|I|): 
\cI \text{~is a countable cover of $A$ by arcs $I$ with $|I|<\eta$}\Bigr\},
\]
and then
\[
\Lambda_h(A):=\lim_{\eta\to0^+}\Lambda_h^\eta(A).
\]
The set function $\Lambda_h$ is an outer measure on $\TT$,
called the \emph{Hausdorff measure} corresponding to the 
measure function $h$. 
If $h(t)=t$ for all $t$, then $\Lambda_h$ is simply the Lebesgue (outer) measure on $\TT$.

We require two general results about Hausdorff measures.
The first of these is a simple consequence of the definitions.

\begin{lemma}\label{L:Hausmeasdefn}
Let $h$ be a measure function and let $\mu$ be a positive
Borel measure on $\TT$. Suppose that there exists $\eta>0$
such that $\mu(I)\le h(|I|)$ for all arcs $I$ with $|I|<\eta$.
Then $\mu(B)\le \Lambda_h(B)$ for all Borel subsets $B\subset\TT$.
\end{lemma}

\begin{proof}
Let $B$ be a Borel subset of $\TT$.
Let $\cI$ be a countable covering of $B$ by arcs $I$ such that $|I|<\eta$.
Then we have
\[
\mu(B)\le\mu(\cup_{I\in\cI}I)\le\sum_{I\in\cI}\mu(I)\le\sum_{I\in\cI}h(|I|).
\]
Taking the infimum over all such coverings $\cI$, we obtain
$\mu(B)\le \Lambda_h^\eta(B)\le\Lambda(B)$.
\end{proof}

The second result is a special case of an observation of Besicovitch \cite[Remark~2]{Be56}.

\begin{lemma}\label{L:Besicovitch}
Let $E$ be a closed subset of $\TT$ such that $|E|=0$.
Then there exists a measure function $h$ such that  
$\Lambda_h(E)=0$ and $\lim_{t\to0^+}h(t)/t=\infty$.
\end{lemma}

\begin{proof}
We construct a sequence $(t_n)_{n\ge0}$ recursively as follows.
Set $t_0:=1$. Suppose that $n\ge1$ and that $t_{n-1}$ has already been chosen. 
Since $E$ is a compact set of Lebesgue measure zero,
we can find a \emph{finite} cover $\cI_n$ of $E$ by arcs $I$ with $|I|<t_{n-1}$ such that
\[
\sum_{I\in\cI_n}|I|<4^{-n}.
\]
Set
\[
t_n:=\min\Bigl\{\min_{I\in\cI_n}|I|,\,t_{n-1}/4\Bigr\}.
\]
Clearly, the sequence $(t_n)$ obtained in  this way is strictly decreasing and tends to  $0$.

Define $h:(0,1]\to(0,\infty)$ by 
\[
h(t):=
\min\{2^nt,\, 2^{n-1}t_{n-1}\} \quad( t\in(t_n,t_{n-1}],\, n\ge1).
\]
Clearly $h$ is continuous and increasing on interval $(t_n,t_{n-1}]$.
Also, since $t_{n-1}\ge 4t_n$, we have
\[
\lim_{t\to t_n^+}h(t)=\min\{2^nt_n,2^{n-1}t_{n-1}\}=2^nt_n=h(t_n),
\]
so in fact $h$ is continuous and increasing on the whole of $(0,1]$.
Further, we have
\[
h(t_n)=2^nt_n\le 2^n4^{-n}t_0=2^{-n}\to0 \quad(n\to\infty),
\]
whence $\lim_{t\to0^+}h(t)=0$. Thus $h$ is a measure function.

For each $n\ge1$, every arc $I\in\cI_n$ satisfies $|I|\in[t_n,t_{n-1}]$,
so $h(|I|)\le 2^n|I|$, and hence
\[
\sum_{I\in\cI_n}h(|I|)\le 2^n\sum_{I\in\cI_n}|I|\le 2^n4^{-n}=2^{-n}.
\]
It follows that $\Lambda_h^{t_{n-1}}(E)\le 2^{-n}$ for all $n$,
whence $\Lambda_h(E)=0$.

Finally, if $t\in[t_n,t_{n-1}]$, then
\[
\frac{h(t)}{t}=\min\Bigl\{\frac{2^nt}{t},\,\frac{2^{n-1}t_{n-1}}{t}\Bigr\}
\ge\min\{2^n,2^{n-1}\}=2^{n-1}.
\]
Hence $h(t)/t\to\infty$ as $t\to0^+$.
\end{proof}

We also need an estimate for singular inner functions.
As before, we consider $\theta$ of the form
\[
\theta(z)=\exp\Bigl(-\int_\TT\frac{e^{it}+z}{e^{it}-z}\,d\nu(e^{it})\Bigr)
\quad(z\in\DD),
\]
where $\nu$ is a finite positive singular measure on $\TT$,
and we set
\[
m_\theta(r):=\inf_{|z|=r}|\theta(z)| \quad(0\le r<1).
\]

\begin{lemma}\label{L:innerest}
With the above notation, we have
\[
-\log m_\theta(1-\eta)\ge \frac{1}{(\pi+1)^2}\sup_{|I|=\eta}\frac{\nu(I)}{|I|}
\quad(0<\eta\le1),
\]
where the supremum is taken over all arcs $I$ with $|I|=\eta$.
\end{lemma}

\begin{proof}
Let $\eta\in(0,1], r\in[0,1)$ and $e^{is}\in\TT$. A simple calculation gives
\begin{align*}
-\log|\theta(re^{is})|
&=\int_{\TT}\frac{1-r^2}{|e^{it}-re^{is}|^2}\,d\nu(e^{it})\\
&\ge\int_{|t-s|\le\pi\eta}\frac{1-r^2}{|e^{it}-re^{is}|^2}\,d\nu(e^{it})\\
&=\int_{|t-s|\le\pi\eta}\frac{(1+r)(1-r)}{|e^{i(t-s)}-1+1-r|^2}\,d\nu(e^{it})\\
&\ge\frac{1-r}{|\pi\eta+1-r|^2}\nu([s-\pi\eta,s+\pi\eta]).
\end{align*}
Taking the supremum of both sides over all $e^{is}\in\TT$, we 
deduce that
\[
-\log m_\theta(r)\ge \frac{1-r}{|\pi\eta+1-r|^2}\sup_{|I|=\eta}\nu(I).\]
Finally, setting $r:=1-\eta$, we obtain
\[
-\log m_\theta(1-\eta)\ge \frac{\eta}{|\pi\eta+\eta|^2}\sup_{|I|=\eta}\nu(I)
=\frac{1}{(\pi+1)^2}\sup_{|I|=\eta}\frac{\nu(I)}{|I|}.\qedhere
\]
\end{proof}

\begin{proof}[Proof of Theorem~\ref{T:key}]
By Lemma~\ref{L:Besicovitch}, there exists a measure
function $h$ such that $\Lambda_h(E)=0$ and 
$h(t)/t\to\infty$ as $t\to0^+$. 
For $n\ge1$ set
\[
t_n:=\inf\Bigl\{t\in(0,1):\frac{h(t)}{-t\log(1-t)}\le n(\pi+1)^2\Bigr\}.
\]
Note that $h(t)/(-t\log(1-t))\to\infty$ as $t\to0^+$,
so the sequence $(t_n)$ is strictly decreasing and 
tends to zero. We shall prove that the sequence
$\epsilon_n:=(1-t_{n+1})^{n/2}$ satisfies the conclusion of the theorem.

First of all, we show that $\epsilon_n\to0$. To see this, note that,
by continuity, we have
\[
\frac{h(t_n)}{-t_n\log(1-t_n)}=n(\pi+1)^2 \quad(n\ge1).
\]
Hence
\[
\log\epsilon_{n-1}=\frac{(n-1)}{2}\log(1-t_{n})
=-\frac{1}{(\pi+1)^2}\frac{n-1}{2n}\frac{h(t_n)}{t_n}\to-\infty \quad (n\to\infty),
\]
which implies that $\epsilon_n\to0$.

Now let $\nu$ be a positive finite measure on $E$
and let $\theta$ be the inner function defined by \eqref{E:sing}.
Since $\nu(E)>0$ and $\Lambda_h(E)=0$,
Lemma~\ref{L:Hausmeasdefn} implies that there exist arcs $I_k$
with $|I_k|\to0$ such that $\nu(I_k)>h(|I_k|)$ for all $k$.
Set $\eta_k:=|I_k|$. 
Suppose that $n$ and $k$ are such that
$t_{n+1}\le \eta_k< t_n$. Then,
using Lemma~\ref{L:innerest}, we have
\begin{align*}
-\log m_\theta(1-\eta_k)
&\ge \frac{1}{(\pi+1)^2}\frac{\nu(I_k)}{|I_k|}\\
&\ge \frac{1}{(\pi+1)^2}\frac{h(\eta_k)}{\eta_k}\\
&> -n\log(1-\eta_k),
\end{align*}
the last inequality because $\eta_k< t_n$.
Thus $m_\theta(1-\eta_k)<(1-\eta_k)^n$,
which, together with \eqref{E:deltanest}, implies that 
\[
\delta_n<(1-\eta_k)^n\le (1-t_{n+1})^n=\epsilon_n^2.
\]
To summarize, we have shown that $\delta_n<\epsilon_n^2$
for each $n$ for which there exists an $\eta_k\in[t_{n+1},t_n)$.
Since $\eta_k\to0$ as $k\to\infty$, there are infinitely many such~$n$.
It follows that $\liminf_{n\to\infty}\delta_n/\epsilon_n\le \liminf_{n\to\infty}\epsilon_n=0$.
\end{proof}


\section{Return to Theorems~\ref{T:main} and \ref{T:converse}}\label{S:return}

We now return to the Theorems~\ref{T:main} and \ref{T:converse} stated in the introduction,
and see what can be said about them in the light of the results developed above.
First of all, we give a  proof of Theorem~\ref{T:converse}.

\begin{proof}[Proof of Theorem~\ref{T:converse}]
Let $E$ be a closed subset of $\TT$ such that $|E|>0$,
and let $(u_n)$ be a positive sequence such that $u_n\to\infty$.
We take $T:=S_\theta$,
where $\theta$ is a  singular inner function.
By Theorems~\ref{T:defect} and \ref{T:negpowers}, 
such an operator automatically satisfies 
$\rank(I-S_\theta^*S_\theta)=1$ and $\sup_n\|S_\theta^{-n}\|=\infty$.
It remains to be shown that $\theta$ can be chosen in such a way that, 
in addition,
$\sigma(S_\theta)\subset E$ and $\|S_\theta^{-n}\|=O(u_n)$.

Note that $x\mapsto x^{-2}\log(2/x)$ is decreasing and bounded by $2/x^3$ on $(0,1)$. Thus, since $u_n\to\infty$, 
there exists a positive sequence $\eta_n\to0$ such that
\[
A\frac{\log(2/\eta_n)}{\eta_n^2}\le u_n \quad(n\ge1),
\]
where $A$ is the constant in \eqref{E:estimates}.
By Theorem~\ref{T:slowdecay}, 
there exists a finite positive singular measure $\nu$ on $E$ such that, 
if $\theta$ is the singular inner function defined by \eqref{E:sing}, 
then $\delta_n(\theta)\ge\eta_n$ for all $n$.  
By Theorem~\ref{T:LM} we then have 
$\sigma(S_\theta)=\supp\nu\subset E$,
and by Corollary~\ref{C:estimates}
\[
\|S_\theta^{-n}\|\le A\frac{\log(2/\eta_n)}{\eta_n^2}\le u_n \quad(n\ge1).
\]
Thus $S_\theta$ has all the required properties, and the proof is complete.
\end{proof}

We can also prove a restricted version of Theorem~\ref{T:main},
applied just to compressed shifts $S_\theta$,
which we formulate as follows.

\begin{theorem}
Let $E$ be a closed subset of $\TT$ such that $|E|=0$.
Then there exists a positive sequence $u_n\to\infty$
such that, if $S_\theta$ is any compressed shift operator with $\sigma(S_\theta)\subset E$, then
\[
\limsup_{n\to\infty}\frac{\|S_\theta^{-n}\|}{u_n}=\infty.
\]
\end{theorem}

\begin{proof}
Let $(\epsilon_n)$ be a sequence satisfying the conclusion
of Theorem~\ref{T:key}. Set $u_n:=1/\epsilon_n$.
Then $(u_n)$
is a positive  sequence such that $u_n\to\infty$.
We shall show that this sequence satisfies the conclusion of the theorem.

Let $S_\theta$ be a compressed shift such that $\sigma(S_\theta)\subset E$. 
By Theorem~\ref{T:LM}, the inner function $\theta$ is singular 
and has the form \eqref{E:sing} for some positive finite measure $\nu$ on $E$. 
By Corollary~\ref{C:estimates}, we have
\[
\|S_\theta^{-n}\|\ge \frac{1}{2}\Bigl(\frac{1}{\delta_n(\theta)}-1\Bigr) \quad(n\ge1).
\]
Also, from Theorem~\ref{T:key}, we have
\[
\liminf_{n\to\infty}\frac{\delta_n(\theta)}{\epsilon_n}=0.
\]
Combining these inequalities, we obtain that
\[
\limsup_{n\to\infty}\frac{\|S_\theta^{-n}\|}{u_n}
\ge \limsup_{n\to\infty}\frac{\epsilon_n}{2}\Bigl(\frac{1}{\delta_n(\theta)}-1\Bigr)
=\frac{1}{2}\limsup_{n\to\infty}\frac{\epsilon_n}{\delta_n(\theta)}=\infty.\qedhere
\]
\end{proof}

Our goal in the rest of the paper is to prove Theorem~\ref{T:main}
for general Hilbert-space contractions.
We shall do this by exploiting the Sz.-Nagy--Foias functional-model theory for contractions.
These functional models are essentially compressed shifts, where now the inner function $\theta$
is allowed to be operator-valued. In the next section we set up the basic background, 
and in the section that follows we deduce Theorem~\ref{T:main}


\section{Functional models}\label{S:functmodels}

Our treatment follows closely that of  \cite[Chapter~V]{Be88} and \cite[Chapter~VI]{SFBL10}.

\subsection{Compressed shifts for operator-valued inner functions}

In this section, $\cF$ and $\cF'$ denote fixed separable Hilbert spaces.
We write $\cL(\cF,\cF')$ for the Banach space of bounded linear operators from $\cF$ into $\cF'$.
Also we set $H^2(\cF):=H^2\otimes\cF$
(with the Hilbert-space tensor product). The shift $S_\cF:H^2(\cF)\to H^2(\cF)$ is defined by
$S_\cF:=S\otimes I_\cF$, where $S$ denotes the usual shift on~$H^2$. Explicitly, we have
\[
(S_\cF f)(\lambda)=\lambda f(\lambda) \quad(f\in H^2(\cF),\,\lambda\in\DD).
\]
 
A function $\Theta\in H^\infty(\cL(\cF,\cF'))$ is said to be \emph{inner} 
if $\Theta(\zeta)$ is an isometry for a.e.\ $\zeta\in\TT$. 

Let $\Theta\in H^\infty(\cL(\cF,\cF'))$ be an inner function. 
The \emph{compressed shift} associated to~$\Theta$ is the operator 
$S_\Theta$ acting on $\cK_\Theta:=H^2(\cF')\ominus \Theta H^2(\cF)$ defined by
\[
S_\Theta:=P_{\cK_\Theta}S_{\cF'}|_{\cK_\Theta}.
\]
Clearly $S_\Theta$ is a contraction operator on $\cK_\Theta$.

\begin{theorem}
Let $\Theta\in H^\infty(\cL(\cF,\cF'))$ be an inner function. 
Then $S_\Theta^{*n}\to0$  strongly as $n\to\infty$.
\end{theorem}

\begin{proof}
We remark that, just as in the scalar case, $S_\Theta^*=S_{\cF'}^*|_{\cK_\Theta}$.
Hence, for all $f\in\cK_\Theta$,
\[
S_\Theta^{*n}f=S_{\cF'}^{*n}f\to0 \quad(n\to\infty).\qedhere
\]
\end{proof}

The next result is the analogue, in this context, of the Liv\v sic--Moeller theorem,
Theorem~\ref{T:LM}.

\begin{theorem}\label{T:LM2}
Let $\Theta\in H^\infty(\cL(\cF,\cF'))$ be an inner function. 
Then the spectrum and point spectrum of $S_\Theta$ are given respectively by
\[
\sigma(S_\Theta)=Y_{\Theta}\cup Z_{\Theta}
\quad\text{and}\quad
\sigma_p(S_\Theta)=Z_{\Theta}^0,
\]
where:
\begin{itemize}
\item $Y_{\Theta}$ is the complement in $\TT$ of 
the union of open arcs where $\Theta$ is regular and unitary-valued,
\item $Z_{\Theta}$ is the set of $\lambda\in\DD$ such that $\Theta(\lambda)$ is not invertible,
\item $Z_{\Theta}^0$ is the set of $\lambda\in\DD$ such that $\Theta(\lambda)$ is not injective.
\end{itemize}
\end{theorem}

\begin{proof}
This follows by combining \cite[p.265, Theorem 4.1]{SFBL10} 
and \cite[p.116, Proposition~1.18]{Be88}.
\end{proof}

\begin{corollary}\label{C:LM2}
If $\sigma(S_\Theta)$ is a proper subset of $\TT$,
then $\Theta$ has a holomorphic extension to $\CC_\infty\setminus\sigma(S_\Theta)$.
This extension satisfies 
\[
\Theta(1/\overline{\lambda})^*
=\Theta(\lambda)^{-1} \qquad(\lambda\in\CC_\infty\setminus\sigma(S_\Theta)).
\]
\end{corollary}

\begin{proof}
The function $\lambda\mapsto\Theta(1/\overline{\lambda})^*$ 
is holomorphic on $\CC_\infty\setminus\overline{\DD}$.
Moreover, by Theorem~\ref{T:LM2}, it continues analytically across
each point of $\TT\setminus\sigma(S_\Theta)$
and coincides with $\Theta(\lambda)^{-1}$ at these points.
It therefore provides an analytic continuation of $\Theta(\lambda)^{-1}$
to the whole of $\CC_\infty\setminus\sigma(S_\Theta)$.
\end{proof}

Let $\Theta\in H^\infty(\cL(\cF,\cF'))$ be an inner function. 
Given $\Phi\in H^\infty(\cL(\cF'))$, 
we write $A_\Phi^\Theta$ for the truncated Toeplitz operator on 
$\cK_\Theta$ defined by
\[
A_{\Phi}^\Theta f:=P_{\cK_\Theta}(\Phi f) \quad(f\in\cK_\Theta).
\]
The next result is a version of the commutant lifting theorem.
In particular, it yields a formula for $\|A_\Phi^\Theta\|$
which is the vector-valued analogue of Theorem~\ref{T:norm}.

\begin{theorem}\label{T:commutant}
Let $\Theta\in H^\infty(\cL(\cF,\cF'))$ be an inner function.
If $X\in\cL(\cK_\Theta)$ and
$XS_\Theta=S_\Theta X$,
then $X=A_\Phi^\Theta$, where $\Phi\in H^\infty(\cL(\cF'))$ 
with $\Phi\Theta H^2(\cF)\subset\Theta H^2(\cF)$
and $\|\Phi\|_{H^\infty(\cL(\cF'))}=\|X\|$.
\end{theorem}

\begin{proof}
See \cite[p.118, Proposition 1.24]{Be88}.
\end{proof}

We use this result to obtain a lower bound for $\|S_{\Theta}^{-n}\|$ when $S_\Theta$ is invertible,
which corresponds to the fact that $\Theta(0)$ is invertible.

\begin{theorem}\label{T:opestimates}
Let $\Theta\in H^\infty(\cL(\cF,\cF'))$ be an inner function. Assume that $S_\Theta$ is invertible.
Then
\[
\|S_{\Theta}^{-n}\|\ge \frac{1}{2}\Bigl(\frac{1}{{\delta_n(\Theta)}}-1\Bigr) \quad(n\ge1),
\]
where 
\begin{equation}\label{E:deltan2}
\delta_n(\Theta):=
\inf\Bigl\{\max\{|\lambda|^n,\,\|\Theta(\lambda)^*x\|\}:~ \lambda\in\DD,\ x\in\cF',\ \|x\|=1\Bigr\}.
\end{equation}
\end{theorem}

\begin{proof}
Fix $n\ge1$. Clearly $S_\Theta^{-n}$ commutes with $S_\Theta$, 
so by Theorem~\ref{T:commutant},
we can write $S_{\Theta}^{-n}=A_\Phi^\Theta$, where $\Phi\in H^\infty(\cL(\cF'))$ with 
$\Phi\Theta H^2(\cF)\subset\Theta H^2(\cF)$ and $\|\Phi\|_\infty=\|S_\Theta^{-n}\|$.
Moreover, since we have $S_\Theta^nS_\Theta^{-n}=I_{\cK_\Theta}$,
it follows that
\[
P_{\cK_\theta}(S_{\cF'}^n\Phi-I_{H^2(\cF')})=0,
\]
which implies that
\[
S_{\cF'}^n\Phi-I_{H^2(\cF')}=\Theta\Xi,
\]
where $\Xi\in H^\infty(\cL(\cF',\cF))$ (see \cite[p.119]{Be88}).

For all $x\in\cF'$ and $\lambda\in\DD$, we have
\begin{equation}\label{E:keyeqn}
\lambda^n \Phi(\lambda)x-x=\Theta(\lambda)\Xi(\lambda)x.
\end{equation}
Taking the $\cF'$-inner product with $x$, we deduce that
\begin{align*}
\langle x,x\rangle_{\cF'}
&=\lambda^n\langle \Phi(\lambda) x,x\rangle_{\cF'}-\langle\Theta(\lambda)\Xi(\lambda)x,x\rangle_{\cF'}\\
&=\lambda^n\langle \Phi(\lambda) x,x\rangle_{\cF'}-\langle\Xi(\lambda)x,\Theta(\lambda)^*x\rangle_\cF,
\end{align*}
whence, if $\|x\|=1$, then
\begin{equation}\label{E:keyeqn2}
1\le |\lambda|^n\|\Phi\|_\infty+\|\Xi\|_\infty\|\Theta(\lambda)^*x\|.
\end{equation}
Now, taking radial limits in \eqref{E:keyeqn} as $\lambda\to\zeta\in\TT$, 
and using the fact that $\Phi(\zeta)$ is an isometry
for a.e.\ $\zeta\in\TT$, we obtain
$\|\zeta^n\Phi(\zeta)x-x\|=\|\Xi(\zeta)x\|$ for a.e.\ $\zeta\in\TT$ and every $x\in\cF'$.
In particular, $\|\Xi\|_\infty\le \|\Phi\|_\infty+1$.
Feeding this information back into \eqref{E:keyeqn2}, we obtain
\begin{align*}
1&\le |\lambda|^n\|\Phi\|_\infty+(\|\Phi\|_\infty+1)\|\Theta(\lambda)^*x\|\\
&\le \max\Bigl\{|\lambda|^n,\ \|\Theta(\lambda)^*x\|\Bigr\}\Bigl(2\|\Phi\|_\infty+1\Bigr).
\end{align*}
As this  holds for all $\lambda\in\DD$ and all $x\in\cF'$ with $\|x\|=1$,
we deduce that $1\le \delta_n(\Theta)(2\|\Phi\|_\infty+1)$,
where $\delta_n(\Theta)$ is defined by \eqref{E:deltan2}.
Thus $\|\Phi\|_\infty\ge (1/2)(1/\delta_n(\Theta)-1)$. Finally, as $\|S_\Theta^{-n}\|= \|\Phi\|_\infty$,
the result follows.
\end{proof}

Here are some alternative formulas for $\delta_n(\Theta)$.

\begin{proposition}\label{P:alternatives}
Let $\Theta\in H^\infty(\cL(\cF,\cF'))$ be an inner function, and let $\delta_n(\Theta)$ be given by
\eqref{E:deltan2}.
\begin{enumerate}[\normalfont(i)]
\item If $\sigma(S_\Theta)\subset\TT$, then
$1/\delta_n(\Theta)=\sup_{|\lambda|<1}\min\bigl\{|\lambda|^{-n},\ \|\Theta(\lambda)^{-1}\|\bigr\}.$
\item If $\sigma(S_\Theta)\subsetneq\TT$, then
$1/\delta_n(\Theta)=\sup_{|\lambda|>1}\min\bigl\{|\lambda|^{n},\ \|\Theta(\lambda)\|\bigr\}.$
\end{enumerate}
\end{proposition}

\begin{proof}
(i) By Theorem~\ref{T:LM2}, if $\sigma(S_\Theta)\subset\TT$, then $\Theta(\lambda)$ is invertible
for all $\lambda\in\DD$. In this case, we have
\[
\inf\{\|\Theta(\lambda)^*x\|:\|x\|=1\}=1/\|(\Theta(\lambda)^*)^{-1}\|=1/\|\Theta(\lambda)^{-1}\|.
\]
The result follows upon feeding this information into \eqref{E:deltan2}.

(ii) Combine part~(i) with Corollary~\ref{C:LM2}.
\end{proof}

The next result is an operator-valued analogue 
of Theorem~\ref{T:key}.
An  operator $T\in \cL(\cF,\cF')$ is said to be
\emph{purely contractive} if 
\[
\|Tf\|<\|f\| \quad(f\in\cF,~f\ne0).
\]

\begin{theorem}\label{T:partial}
Let $E$ be a closed subset of $\TT$ such that $|E|=0$.
Then there exists a positive sequence $\epsilon_n\to0$
with the following property.
If $\cF,\cF'$ are finite-dimensional Hilbert spaces, 
if $\Theta:(\CC_\infty\setminus E)\to\cL(\cF,\cF')$ is a  holomorphic function such that
$\Theta|_{\DD}$ is an inner function,
and if $\Theta(\lambda)$ is purely contractive for all $\lambda\in\DD$,
then
\[
\liminf_{n\to\infty}\frac{\delta_n(\Theta)}{\epsilon_n}=0.
\]
\end{theorem}

\begin{proof}
The case $\dim(\cF)=\dim(\cF)=1$ is just Theorem~\ref{T:key}.
Let $(\epsilon_n)$ be the sequence furnished by that theorem.
We shall show that it also works whenever $\dim(\cF),\dim(\cF')<\infty$. 

Fix $\zeta_0\in\TT\setminus E$. Then $\Theta(\zeta_0)$
is a unitary operator from $\cF$ onto $\cF'$. 
Thus, replacing $\Theta$ by $\Theta(\zeta_0)^{-1}\Theta$,
we can suppose from the outset that $\cF'=\cF$.

Let $\cF$ be a Hilbert space of dimension $N$,
and let $\Theta:(\CC_\infty\setminus E)\to\cL(\cF)$ 
be a non-constant holomorphic function such that
$\Theta|_{\DD}$ is an inner function. 
Note that, since $\Theta(\lambda)\Theta(1/\overline{\lambda})^*$
is holomorphic on $\CC\setminus E$ and 
$\Theta(\lambda)\Theta(1/\overline{\lambda})^*
=\Theta(1/\overline{\lambda})^*\Theta(\lambda)=I$ 
for all $\lambda\in\TT\setminus E$, 
the same equality persists for all $\lambda\in\CC\setminus E$.
In particular, $\Theta(\lambda)$ is invertible for all $\lambda\in\CC\setminus E$,
and thus $\sigma(S_\Theta)\subset E$.

Consider the function $\Delta:\DD\to\CC$ defined by
\[
\Delta(\lambda):=\det(\Theta(\lambda))^{1/N} \quad(\lambda\in \DD).
\]
Since $\Theta(\lambda)$ is invertible, we have $\det(\Theta(\lambda))\ne0$,
so there exists a holomorphic choice of $N$-th root on $\DD$ (any such choice will do).
Thus $\Delta$ is holomorphic and nowhere zero on  $\DD$, 
and has boundary values of modulus $1$ on $\lambda\in\TT\setminus E$. 
Also $|\Delta(\lambda)|<1$ for all $\lambda\in \DD$,
since $\Theta(\lambda)$ is purely contractive for all $\lambda\in\DD$.
In particular, $\Delta$ is not a constant function.
Thus $\Delta$ is a singular inner function with  singular measure supported on $E$.

By Theorem~\ref{T:key}, we have
\[
\liminf_{n\to\infty}\frac{\delta_n(\Delta)}{\epsilon_n}=0.
\]
Now, if $A$ is any $N\times N$ matrix, then $|\det(A)|\le\|A\|^N$. 
Applying this inequality with $A=\Theta(\lambda)^{-1}$,
we deduce, using Proposition~\ref{P:alternatives}(i), that
\begin{align*}
1/\delta_n(\Theta)
&=\sup_{|\lambda|<1}\min\Bigl\{|\lambda|^{-n},\ \|\Theta(\lambda)^{-1}\|\Bigr\}\\
&\ge\sup_{|\lambda|<1}\min\Bigl\{|\lambda|^{-n},\ |\det(\Theta(\lambda)^{-1})|^{1/N}\Bigr\}\\
&=\sup_{|\lambda|<1}\min\Bigl\{|\lambda|^{-n},\ |\Delta(\lambda)|^{-1}\Bigr\}=1/\delta_n(\Delta).
\end{align*}
Hence $\delta_n(\Theta)\le\delta_n(\Delta)$ for all $n$, and the result follows.
\end{proof}


\subsection{The model theorem}\label{S:modeltheorem}

Let $H$ be a separable Hilbert space, and let $T\in\cL(H)$ be a contraction.

The \emph{defect operators} of $T$ and $T^*$ are respectively
\[
D_T:=(I-T^*T)^{1/2}
\quad\text{and}\quad
D_{T^*}:=(I-TT^*)^{1/2}.
\]
The corresponding \emph{defect spaces} are
\[
\cD_T:=\overline{D_T(H)}
\quad\text{and}\quad
\cD_{T^*}:=\overline{D_{T^*}(H)}.
\]
Since $D_T$ is a limit of polynomials in $(I-T^*T)$ (and likewise for $D_{T^*}$),
the relation $T(I-T^*T)=(I-TT^*)T$ implies that
\[
TD_T=D_{T^*}T.
\]
In particular, we have $T(\cD_T)\subset \cD_{T^*}$.

The \emph{characteristic function} of $T$ is the function
$\Theta_T:\DD\to\cL(\cD_T,\cD_{T^*})$ defined by
\[
\Theta_T(\lambda)x:= 
\Bigl(-T+\lambda D_{T^*}(I-\lambda T^*)^{-1}D_T\Bigr)x 
\quad (x\in\cD_T,\ \lambda\in\DD).
\]
Clearly $\Theta_T$ is holomorphic on $\DD$. It extends, via the same formula, to
be holomorphic on $\{\lambda\in\CC_\infty:1/\overline{\lambda}\notin\sigma(T)\}$.
In particular, if $\sigma(T)\subset\TT$, 
then $\Theta_T$ is holomorphic on $\CC_\infty\setminus\sigma(T)$.

The following result is a version of the model theorem for contractions.

\begin{theorem}\label{T:model}
Let $T$ be a contraction on a separable Hilbert space 
such that $T^{*n}\to0$ strongly as $n\to\infty$.
Then $\Theta_T$ is a non-constant inner function, 
and $T$ is unitarily equivalent to $S_{\Theta_T}$.
\end{theorem}

\begin{proof}
See \cite[p.115, Corollary 1.14]{Be88}.
\end{proof}
 
\begin{theorem}\label{T:pc}
If $T$ is a non-unitary contraction, then, for all $\lambda\in\DD$, 
the operator $\Theta_T(\lambda)$ is purely contractive, i.e.,
\[
\|\Theta_T(\lambda)f\|<\|f\| 
\quad(f\in\cD_T,\ f\ne0).
\]
\end{theorem}

\begin{proof}
See \cite[p.245]{SFBL10}.
\end{proof}


\section{Proof of Theorem~\ref{T:main}}\label{S:mainproof}

We begin with some preliminary reductions.

\begin{lemma}\label{L:reduction}
Let $T$ be a non-unitary contraction on a Hilbert space $H$.
Then there exists a separable, closed, $T$-reducing subspace $H_1$ of $H$ 
such that $T_1:=T|_{H_1}$
is completely non-unitary. 
If $(I-T^*T)$ has finite rank, then so does $(I-T_1^*T_1)$.
Furthermore, if $\sigma(T)\subsetneq\TT$, then $\sigma(T_1)\subset\sigma(T)$ 
and $\|T_1^{-n}\|\le\|T^{-n}\|$ for all $n\ge1$.
\end{lemma}

\begin{proof}
By Langer's Lemma \cite[p.7]{Ni86}, 
we can decompose $H$ as the orthogonal direct sum $H'\oplus H''$, 
where $H',H''$ are closed,  $T$-invariant subspaces such that 
$T|_{H'}$ is unitary and $T|_{H''}$ is completely non-unitary. 
Since $T$ is non-unitary, we have $H''\ne0$. 
Pick a non-zero $x\in H''$,
and let $H_1$ be the closure of the set $\{p(T,T^*)x\}$ 
as $p(T,T^*)$ runs through all non-commutative polynomials in $T,T^*$.
Then $H_1$ is closed in $H$, separable, $T$-reducing, 
and $T_1:=T|_{H_1}$ is completely non-unitary. 

Since $H_1$ is $T$-reducing, we have $T_1^*=T^*|_{H_1}$.
It follows easily that $\rank(I-T_1^*T_1)\le \rank(I-T^*T)$.

Suppose now that $\sigma(T)\subsetneq\TT$. 
Let $\lambda\in\CC\setminus\sigma(T)$. 
By Runge's theorem, there exists a 
sequence of polynomials $p_n(z)$ converging uniformly 
to $1/(z-\lambda)$ for $z$ in a neighbourhood 
of $\sigma(T)$, which implies that $p_n(T)\to(T-\lambda I)^{-1}$ in norm. 
In particular $(T-\lambda I)^{-1}H_1\subset H_1$. 
This shows that $(T_1-\lambda I)$ is invertible in $\cL(H_1)$, 
with inverse $(T-\lambda I)^{-1}|_{H_1}$.
Thus $\sigma(T_1)\subset \sigma(T)$.
It also follows that $\|T_1^{-n}\|=\|T^{-n}|_{H_1}\|\le\|T^{-n}\|$ for all $n\ge1$, as claimed.
\end{proof}

\begin{lemma}\label{L:C00}
Let $T$ be a completely non-unitary contraction on a separable Hilbert space,
and suppose that $\sigma(T)\cap\TT$ has Lebesgue measure zero. Then both $T^n$ and $T^{*n}$ 
converge strongly to zero as $n\to\infty$.
\end{lemma}

\begin{proof}
See \cite[Chapter~II, Proposition~6.7]{SFBL10}.
\end{proof}

Finally, we assemble all the pieces. 

\begin{proof}[Proof of Theorem~\ref{T:main}]
Let $E$ be a closed subset of $\TT$ such that $|E|=0$.
Let $(\epsilon_n)$ be the sequence furnished by Theorem~\ref{T:partial}, 
and set $u_n:=1/\epsilon_n$.
Then $(u_n)$ is a positive  sequence such that $u_n\to\infty$.
Let $T$ be a non-unitary Hilbert-space contraction such that 
$\sigma(T)\subset E$ and $\rank(I-T^*T)<\infty$.
We shall prove that
\[
\limsup_{n\to\infty}\frac{\|T^{-n}\|}{u_n}=\infty.
\]

By Lemma~\ref{L:reduction}, it suffices to consider the case where,
in addition, $T$ is a completely non-unitary contraction acting on a 
separable Hilbert space. By Lemma~\ref{L:C00}
$T^{*n}\to0$ strongly. Define $\cD_T,\cD_{T^*}$ and
$\Theta_T$ as in \S\ref{S:modeltheorem}. 
Since $(I-T^*T)$ has finite rank, it follows that $\dim\cD_T=\dim\cD_{T^*}<\infty$. 
By Theorem~\ref{T:model},  $\Theta_T:\DD\to\cL(\cD_T,\cD_{T^*})$
is an inner function, and $T$ is unitarily equivalent to the corresponding
compressed shift operator $S_{\Theta_T}$.
Furthermore, by Theorem~\ref{T:pc}, $\Theta_T(\lambda)$ is
purely contractive for all $\lambda\in\DD$.
Therefore, by Theorem~\ref{T:opestimates},  we have
\[
\|S_{\Theta_T}^{-n}\|
\ge \frac{1}{2}\Bigl(\frac{1}{\delta_n(\Theta_T)}-1\Bigr) 
\quad(n\ge1),
\]
where 
$\delta_n(\cdot)$ is defined as in that theorem.
From Theorem~\ref{T:partial} we  also have
\[
\liminf_{n\to\infty}\frac{\delta_n(\Theta_T)}{\epsilon_n}=0.
\]
Combining all these various facts, we deduce that
\[
\limsup_{n\to\infty}\frac{\|T^{-n}\|}{u_n}
=\limsup_{n\to\infty}\epsilon_n\|S_{\Theta_T}^{-n}\|
\ge \limsup_{n\to\infty}\frac{\epsilon_n}{2}\Bigl(\frac{1}{\delta_n(\Theta_T)}-1\Bigr)=\infty,
\]
as was to be shown.
\end{proof}


\section{Conclusion}\label{S:conclusion}

As mentioned in the introduction, Theorem~\ref{T:main}
stops short of establishing Conjecture~\ref{Cj:Esterle}
because of the additional hypothesis that
$\rank(I-T^*T)<\infty$. This hypothesis is needed in the proof
in order to use the determinant trick in Theorem~\ref{T:partial}.
The full conjecture would be proved 
if we could establish the following general version of Theorem~\ref{T:partial},
which we  formulate as a conjecture.

\begin{conjecture}\label{Cj:key}
Let $E$ be a closed subset of $\TT$ such that $|E|=0$.
Then there exists a positive sequence $\epsilon_n\to0$
with the following property.
If $\cF,\cF'$ are separable  Hilbert spaces, 
if $\Theta:(\CC_\infty\setminus E)\to\cL(\cF,\cF')$ is a  holomorphic function such that
$\Theta|_{\DD}$ is an inner function,
and if $\Theta(\lambda)$ is purely contractive for all $\lambda\in\DD$,
then
\[
\liminf_{n\to\infty}\frac{\delta_n(\Theta)}{\epsilon_n}=0.
\]
\end{conjecture}

This remains unproved. 
The ingredient that we lack  is an explicit formula for $\Theta$ of the type \eqref{E:sing}.


\section*{Acknowledgements}

I am grateful to Ken Davidson for bringing 
reference \cite{Sa67} to my attention.

This work owes a great deal to Jean Esterle. 
It was he who first drew my attention to this subject,  
who proposed Conjecture~\ref{Cj:Esterle} 
and who suggested that it could be approached by studying the properties of inner functions. 
The paper is dedicated to his memory.


\bibliographystyle{amsplain}
\bibliography{biblist.bib}

\end{document}